\documentclass[11pt]{article}
\usepackage{amscd}
\usepackage{amsmath}
\usepackage{latexsym}
\usepackage{amsfonts}
\usepackage{amssymb}
\usepackage{amsthm}
\usepackage{graphicx}

\newtheorem{theorem}{Theorem}[section]
\newtheorem{lemma}{Lemma}[section]

\newtheorem{proposition}{Proposition}[section]
\newtheorem{definition}{Definition}[section]

\newtheorem{corollary}{Corollary}[section]
\newtheorem{remark}{Remark}[section]
\newtheorem{conjecture}{Conjecture}[section]

\begin{document}

 \title{ Global wellposedness and scattering for the defocusing energy-critical nonlinear
Schr\"odinger equations of fourth order in dimensions $d\geq9$}
\author{{Changxing Miao$^{\dag}$\ \ Guixiang Xu$^{\dag}$ \ \ and \ Lifeng Zhao $^{\ddag}$}\\
         {\small $^{\dag}$Institute of Applied Physics and Computational Mathematics}\\
         {\small P. O. Box 8009,\ Beijing,\ China,\ 100088}\\
         {\small $^\ddag$ Department of Mathematics, University of Science and Technology of China}\\
         {\small Hefei,\ China,\ 230026}\\
         {\small (miao\_changxing@iapcm.ac.cn, \ xu\_guixiang@iapcm.ac.cn, zhaolifengustc@yahoo.cn) }\\
         \date{}
        }
        \maketitle

\begin{abstract}We consider the defocusing
energy-critical nonlinear Schr\"odinger equation of fourth order
$iu_t+\Delta^2 u=-|u|^\frac{8}{d-4}u$. We prove that any finite
energy solution is global and scatters both forward and backward in
time in dimensions $d\geq9$.
\end{abstract}


\section{Introduction}
In this paper, we will investigate the defocusing energy-critical
Schr\"odinger equation of fourth order, namely,
\begin{equation}\label{fnls} \left\{ \aligned
    iu_t +  \Delta^2 u  & = -|u|^\frac{8}{d-4}u, \quad  \text{in}\  \mathbb{R}^d \times \mathbb{R},\\
     u(0)&=u_0(x), \quad \text{in} \ \mathbb{R}^d.
\endaligned
\right.
\end{equation} The name `energy-critical' refers to the fact that the scaling
symmetry
\begin{equation*} u(t,x)\mapsto
u_\lambda(t,x):=\lambda^\frac{d-4}{2}u(\lambda^4t, \lambda x)
\end{equation*}
leaves both the equation and the energy invariant. The energy of a
solution is defined by
\begin{equation*}
E(u(t))=\frac{1}{2}\int_{\Bbb R^d}|\Delta
u(t,x)|^2dx+\frac{d-4}{2d}\int_{\Bbb R^d}|u(t,x)|^\frac{2d}{d-4}dx
\end{equation*}
and is conserved under the flow. We refer to the Laplacian term in
the formula above as the kinetic energy and to the second term as
the potential energy.
\begin{definition}[Solutions.] A function $u: I\times\Bbb
R^d\rightarrow\Bbb C$ on a non-empty time interval $t_0\in
I\subset\Bbb R$ is a solution to (\ref{fnls}) if it lies in the
class $C_t^0\dot{H}_x^2(K\times\Bbb R^d)\cap
L_{t,x}^\frac{2(d+4)}{d-4}(K\times\Bbb R^d)$ for all compact
$K\subset I$, and obeys the Duhamel formula
\begin{equation*}
u(t)=e^{i(t-t_0)\Delta^2}u(t_0)+i\int_{t_0}^t
e^{i(t-\tau)\Delta^2}F(u(\tau))d\tau
\end{equation*}
for all $t\in I$, where $F(u)=|u|^\frac{8}{d-4}u$. We refer to $I$
as the lifespan of $u$. We say that $u$ is a maximal-lifespan
solution if the solution cannot be extended to any strictly larger
interval. We say that $u$ is a global solution if $I=\Bbb R$.
\end{definition}
\begin{conjecture}\label{cj} Let $d\geq5$ and let $u: I\times\Bbb R^d\rightarrow\Bbb
C$ be a solution to (\ref{fnls}) with finite energy $E$, then
$I=\mathbb{R}$ and
\begin{equation*} \int_\Bbb R\int_{\Bbb
R^d}|u(t,x)|^\frac{2(d+4)}{d-4}dxdt\leq C(E)<\infty.
\end{equation*}
\end{conjecture}
This conjecture has been verified for radial data by B. Pausader
\cite{Pa}. In this paper, we will verify this conjecture for general
data in dimensions $d\geq9$. In fact, we establish the following
theorem:
\begin{theorem}\label{cj} Let $d\geq9$ and let $u: I\times\Bbb R^d\rightarrow\Bbb
C$ be a solution to (\ref{fnls}) with finite energy $E$, then
$I=\mathbb{R}$ and
\begin{equation*} \int_\Bbb R\int_{\Bbb
R^d}|u(t,x)|^\frac{2(d+4)}{d-4}dxdt\leq C(E)<\infty.
\end{equation*}
\end{theorem}
 The ideas and
techniques for fourth order nonlinear Schr\"odinger equations come
from the study of classical nonlinear Schr\"odinger equations. For
the energy critical nonlinear Schr\"odinger equations
\begin{equation}\label{nls} \left\{ \aligned
    iu_t +  \Delta u  & = \lambda|u|^\frac{4}{d-2}u, \quad  \text{in}\  \mathbb{R}^d \times \mathbb{R},\\
     u(0)&=u_0(x), \quad \text{in} \ \mathbb{R}^d,
\endaligned
\right.
\end{equation}
the local well-posedness and global well-posedness for small data
were established by T. Cazenave and F. B. Weissler \cite{CW}
regardless of the sign of $\lambda$. The global well-posedness and
scattering for large data have been extensively studied.

For the defocusing case $\lambda=+1$, J. Bourgain proved global
well-posedness and scattering for radial solution in dimensions
three and four in \cite{borg:scatter}, with the ``induction on
energy'' strategy he invented. Subsequently, G. Grillakis \cite{Gr}
gave a different argument which recovered part of
\cite{borg:scatter}, namely, global existence from smooth, radial,
finite energy data. Later on, T. Tao \cite{tao} generalized the
results of Bourgain to any dimension $d\geq3$ and got bounds on
various spacetime norms of the solution which are exponential type
in the energy, which improved Bourgain's tower type bounds. J.
Colliander, M. Keel, G. Staffilani, H. Takaoka, and T. Tao
\cite{ckstt:gwp} established global well-posedness and scattering
for solutions in energy space in dimension three. The method is
similar in spirit to the induction on energy strategy of Bourgain,
but they performed the induction analysis in both frequency space
and physical space simultaneously, and replaced the Morawetz
inequality by an interaction Morawetz estimate. The principle
advantage of the interaction Morawetz estimate is that it is not
localized in spatial origin and so is better able to handle
nonradial solutions. E. Ryckman and M. Visan extended this results
to dimensions four and higher in \cite{rv}, \cite{vis}.

A new and efficient approach to the energy-critical nonlinear
Schr\"odinger equations was introduced by C. E. Kenig and F. Merle
\cite{kenig-merle}, where they obtained global well-posedness and
scattering for radial data with energy and kinetic energy less than
those of ground state in the focusing case in dimensions $3\leq
d\leq5$. Their arguments work equally well for the defocusing case.
They employed a (concentration) compactness technique in place of
previous localization arguments. They reduced matters to a rigidity
theorem using a concentration compactness argument, with the aid of
localized Virial identity. The radiality enters only at one point in
the proof of the rigidity theorem because of the difficulty in
controlling the motion of spatial translation of global solutions.
Earlier steps in this direction include \cite{BG}, \cite{BV},
\cite{sh1}, \cite{sh2} and \cite{mv}. R. Killip and M. Visan
\cite{kv} improved this result to general solutions in $d\geq5$. The
method is to reduce minimal kinetic energy blow up solutions to
almost periodic solutions modulo symmetries, which match one of the
three scenarios: finite time blowup, low-to-high cascade and
soliton. Then the aim is to eliminate such solutions. The finite
time blowup solutions can be precluded using the method in
\cite{kenig-merle}. For the other two types of solutions, R. Killip
and M. Visan proved that they admit additional regularities, namely,
they belong to $\dot{H}^{-\epsilon}_x$ for some $\epsilon>0$. In
particular, they are in $L_x^2$. Similar ideas have appeared in
\cite{KTV}, \cite{kv} and \cite{KVZ} when dealing with mass-critical
nonlinear Schr\"odinger equations. But a remarkable difficulty comes
from the minimal kinetic energy blowup solution because the kinetic
energy, unlike the energy, is not conserved. Related arguments (for
the cubic NLS in dimension three) appear in \cite{kem2}. The
low-to-high cascade can be precluded by negative regularity and the
conservation of mass. It remains to preclude the soliton. In this
case, one need to control the motion of spatial center function of
the soliton solution, which can be obtained by using the method from
\cite{dhr} and \cite{kem} and the negative regularity. The fist step
is to note that a minimal kinetic energy blowup solution with finite
mass must have zero momentum. A second ingredient needed to control
the motion of spatial center function is a compactness property of
the orbit of $\{u(t)\}$ in $L_x^2$. The argument from \cite{dhr}
gives that the spatial translation is $o(t)$ instead of $O(t)$ given
by simple argument as $t\rightarrow\infty$. Finally the soliton-like
solution is precluded by using a truncated Virial identity. However,
the negative regularity in \cite{kv} cannot be obtained in
dimensions three and four because the dispersion is too weak.

\begin{definition}[Symmetry group]\label{d2}
For any phase $\theta\in \Bbb R/2\pi\Bbb Z$, position $x_0\in\Bbb
R^d$ and scaling parameter $\lambda>0$, we define the unitary
transformation $g_{\theta, x_0, \lambda}: \dot{H}^2(\Bbb
R^d)\rightarrow\dot{H}^2(\Bbb R^d)$ by the formula
\begin{equation*}[g_{\theta, x_0,
\lambda}f](x):=\lambda^{-\frac{d-4}{2}}e^{i\theta}f\big(\lambda^{-1}(x-x_0)\big).
\end{equation*}
We let $G$ be the collection of such transformations. If $u:
I\times\Bbb R^d\rightarrow\Bbb C$ is a function, we define
$T_{g_{\theta, x_0, \lambda}}u: \lambda^4I\times\Bbb
R^d\rightarrow\Bbb C$ where $\lambda^4I:=\{\lambda^4t: t\in I\}$ by
the formula
\begin{equation*}
[T_{g_{\theta, x_0,
\lambda}}u](t,x):=\lambda^{-\frac{d-4}{2}}e^{i\theta}u\big(\lambda^{-4}t,
\lambda^{-1}(x-x_0)\big).
\end{equation*}
\end{definition}
\begin{definition}[Almost periodic solutions]\label{d1} Let $d\geq 5$. A solution $u$ to (\ref{fnls})
with lifespan $I$ is said to be almost periodic modulo $G$ if there
exist functions $N: I\rightarrow\Bbb R^+$, $x: I\rightarrow\Bbb R^d$
and $C: \Bbb R^+\rightarrow\Bbb R^+$ such that for all $t\in I$, and
$\eta>0$,
\begin{equation}\label{equ51}
\int_{|x-x(t)|\geq C(\eta)/N(t)}|\Delta u(t,x)|^2dx\leq\eta
\end{equation}
and
\begin{equation}\label{equ52}
\int_{|\xi|\geq C(\eta)N(t)}|\xi|^4|\hat{u}(t,\xi)|^2d\xi\leq\eta.
\end{equation}
We refer to the function $N$ as the frequency scale function for the
solution $u$, $x$ the spatial center function, and to $C$ as the
compactness modulus function.
\end{definition}
By the Ascoli-Arzela theorem, a family of functions is precompact in
$\dot{H}_x^2$ if and only if it is norm-bounded and there exists a
compactness modulus function $C$ so that
\begin{equation*}
\int_{|x|\geq C(\eta)}|\Delta f(x)|^2dx+\int_{|\xi|\geq
C(\eta)}|\xi|^4|\hat{f}(\xi)|^2d\xi\leq \eta
\end{equation*}
for all functions $f$ in the family. By Sobolev embedding, any
solution $u: I\times\Bbb R^d\rightarrow\Bbb C$ that is almost
periodic modulo $G$ must also satisfy \begin{equation}\label{pot}
\int_{|x-x(t)|\geq C(\eta)/N(t)}|u(t,x)|^\frac{2d}{d-4}dx\leq\eta.
\end{equation}
\begin{remark}
By Ascoli-Arzela theorem, the above definition is equivalent to
either of the following two statements:
\begin{enumerate}
\item The quotient orbit $\Big\{Gu(t): t\in I\Big\}$ is a
precompact set of $G\backslash \dot{H}^2$, where $G\backslash
\dot{H}^2$ is the moduli space of $G$-orbits $Gf:=\{gf: g\in G\}$ of
$\dot{H}^2(\Bbb R^d)$.
\item There exists a compact subset $K$ of $\dot{H}^2$
such that $u(t)\in GK$ for all $t\in I$; equivalently there exists a
group function $g:I\rightarrow G$ and a compact subset $K$ such that
$g^{-1}(t)u(t)\in K$ for any $t\in I$.
\end{enumerate}
\end{remark}
\begin{remark}\label{rem1} A further consequence of compactness
modulo $G$ is the existence of a function $c: \Bbb
R^+\rightarrow\Bbb R^+$ so that
\begin{equation}\label{cc}
\int_{|x-x(t)|\leq c(\eta)/N(t)}|\Delta u(t,x)|^2dx+\int_{|\xi|\leq
c(\eta)N(t)}|\xi|^4|\hat{u}(t,\xi)|^2d\xi\leq\eta
\end{equation}
for all $t\in I$ and $\eta>0$.

In fact, since $K$ is compact in $\dot{H}^2(\Bbb R^d)$, there exists
$c(\eta)$ such that
\begin{equation*}
\sup_{f\in K}\int_{|x|<c(\eta)}|\Delta f|^2dx<\eta.
\end{equation*}
Thus
\begin{equation*}
\int_{|x-x(t)|\leq c(\eta)/N(t)}|\Delta u(t,x)|^2dx=
\int_{|x|<c(\eta)}|\Delta g^{-1}(t)u(t)|^2dx<\sup_{f\in
K}\int_{|x|<c(\eta)}|\Delta f|^2dx<\eta.
\end{equation*}

%if there does not exist such $c(\eta)$, then there exists
%$\varepsilon_n\rightarrow0$ and $t_n$ such that \begin{equation*}
%\int_{|\xi|<\varepsilon_nN(t_n)}|\xi|^4|\hat{u}(t_n,\xi)|^2d\xi>\eta.
%\end{equation*}
%If $N(t_n)<C$, then
%\begin{equation*}
%\int_{|\xi|<\varepsilon_nN(t_n)}|\xi|^4|\hat{u}(t_n,\xi)|^2d\xi\leq\int_{|\xi|<C\varepsilon_n}|\xi|^4|\hat{u}(t_n,\xi)|^2d\xi,
%\end{equation*}
%which, by energy conservation and continuity of integrals,
%approaches zero as $n\rightarrow\infty$. If
%$\lim_{n\rightarrow\infty}N(t_n)=+\infty$, then
%\begin{equation*}
%\eta>\int_{|x-x(t_n)|>C(\eta)/N(t_n)}|\Delta
%u(t_n,x)|^2dx\geq\int_{|\xi|<\varepsilon_nN(t_n)}|\xi|^4|\hat{u}(t,\xi)|^2d\xi>\eta,
%\end{equation*}
%which is absurd. Thus we established there exists $c(\eta)$ such
%that \begin{equation*}\int_{|\xi|\leq
%c(\eta)N(t)}|\xi|^4|\hat{u}(t,\xi)|^2d\xi\leq\eta.
%\end{equation*}
We can prove similarly that there exists $c(\eta)$ such that
\begin{equation*}\int_{|\xi|\leq c(\eta)N(t)}|\xi|^4|\hat{u}(t,\xi)|^2d\xi\leq\eta.
\end{equation*}
\end{remark}
In \cite{mxz3}, we have made a lot of preparations including the
following two theorems:
\begin{theorem}[Reduction to almost periodic solutions, \cite{mxz3}]\label{raps}
Suppose $d\geq5$ is such that Conjecture \ref{cj} failed. Then there
exists a maximal-lifespan solution $u: I\times\Bbb
R^d\rightarrow\Bbb C$ to (\ref{fnls}) such that $E(u)<\infty$, $u$
is almost periodic modulo $G$, and $u$ blows up both forward and
backward in time. Moreover, $u$ has minimal kinetic energy among all
blowup solutions, that is
\begin{equation*}
\sup_{t\in I}\|\Delta u(t)\|_{L^2}<\sup_{t\in J}\|\Delta
v(t)\|_{L^2}
\end{equation*}
for all maximal-lifespan solutions $v: J\times\Bbb
R^d\rightarrow\Bbb C$ that blowup at least one time direction.
\end{theorem}
\begin{theorem}[Three special scenarios for blowup, \cite{mxz3}]\label{enemies}
Fix $d\geq5$ and suppose that Conjecture \ref{cj} fails for this
choice of $d$. Then there exists a minimal kinetic energy,
maximal-lifespan solution $u: I\times\Bbb R^d\rightarrow\Bbb C$,
which is almost periodic modulo symmetries, and obeys
 $$S_I(u)=\int_I \int_{\Bbb
R^d}|u(t,x)|^\frac{2(d+4)}{d-4}dxdt=\infty, \quad E(u)<\infty.$$

We can also ensure that the lifespan $I$ and the frequency scale
function $N: I\rightarrow\Bbb R^+$ match one of the following three
scenarios:
\renewcommand{\labelenumi}{\Roman{enumi}.}
\begin{enumerate}
\item {\rm( Finite time blowup.)} We have that either $|\inf
I|<\infty$ or $\sup I<\infty$.
\item {\rm( Soliton-like solution.)} We have $I=\Bbb R$ and $$N(t)=1,\ \  \text{for all}\ \ t\in\Bbb R.$$
\item {\rm(Low-to-high frequency cascade.)} We have $I=\Bbb R$ and $$\inf_{t\in\Bbb R}N(t)\geq1,\ \
\text{and}\ \ \limsup_{t\rightarrow\infty}N(t)=\infty.$$
\end{enumerate}
\end{theorem}
This paper is devoted to precluding the existence of solutions that
satisfy the criteria in Theorem \ref{enemies}. The argument here is
a direct ``fourth order'' analogue of that in \cite{kv}. The key
step in all three scenarios above is to prove additional regularity,
that is, the solution $u$ lies in $L_x^2$ or better. The finite time
blow up can be precluded using the method of C. E. Kenig, F. Merle
\cite{kenig-merle}, that is, we prove that the $L_x^2$ norm of
$u(t)$ converges to zero as $t$ approaches the finite endpoint.
Since mass is conserved, this implies that $u$ is identically zero.
To preclude the the other two types, we will prove that they have
negative regularities. This is achieved in two stages. First, we
prove that the solution belongs to $L_t^\infty L_x^p$ for certain
values of $p$ less than $2d/(d-4)$. The second step is to upgrade
the decay proved in the first step to $L_x^2$-based spaces. Thus we
can preclude the low-to-high frequency cascade by negative
regularity and the conservation of mass.

To preclude the soliton-like solutions, we adapt a different
argument from \cite{kv} because no Galilean type transformation is
known for the nonlinear Schr\"odinger equations of fourth order. We
first prove that the $L^p_x$ ($1<p<\infty$) norm of soliton solution
is bounded from below. In fact, we can see from the proof that this
is true for any almost periodic solutions. Next, using the negative
regularity for the soliton solution, we derived an interaction
Morawetz estimate. The interaction Morawetz estimate holds only for
soliton (and low-to-high cascade) instead of all actual solutions
here. Moreover, we needn't localize the soliton solution
 in either physics or frequency space as in \cite{ckstt:gwp} because it belongs to $L_t^\infty H^2_x$.
 Finally we prove that some spacetime norm of the soliton is
 infinity, which contradicts the spacetime bound obtained from the
 interaction Morawetz estimate. In addition, this argument can be
 applied to other defocusing Schr\"odinger-type equations
 to preclude the soliton-like solution once one prove that such solution
 admits sufficient regularity.

At last, we will mention that the defocusing assumption is only used
in precluding the soliton. So the negative regularity for
low-to-high cascade and soliton remains true in focusing case. If
one has the Galilean type transformation, then the global
well-posedness and scattering for focusing energy-critical nonlinear
Schr\"odinger equations of fourth order in dimensions $d\geq9$ can
probably be solved using the method in \cite{kv}. The dimension
restriction appears in the proof of the negative regularity because
the dispersion is not strong enough to perform the double Duhamel
trick.
%But for dimension $5\leq d\leq8$, the problem seems quite
%difficult even for defocusing case.

After the paper was finished, we learned that B. Pausader \cite{Pa2}
has obtained independently similar result in dimension $d=8$ and the
high dimensional results can also be obtained using his method.

 The rest of the paper is organized as follows: In Section 2, we
 introduce some notations and preliminaries. Section 3 is devoted to
 deriving a very important property of almost periodic solutions:
 double Duhamel formula. In Section 4, we preclude the finite time
 blow up solutions. In Section 5, we prove the negative regularity
 for low-to-high cascade and soliton. In Section 6, we preclude the
 low-to-high cascade and in Section 7, we kill the soliton.

\section{Notations and preliminaries}
We use $X\lesssim Y$ or $Y\gtrsim X$ whenever $X\leq CY$ for some
constant $C>0$. We use $O(Y)$ to denote any quantity $X$ such that
$|X|\lesssim Y$. We use the notation $X\sim Y$ whenever $X\lesssim
Y\lesssim X$. The fact that these constants depend upon the
dimension $d$ will be suppressed. If $C$ depends upon some
additional parameters, we will indicate this with subscripts; for
example, $X\lesssim_u Y$ denotes the assertion that $X\leq C_u Y$
for some $C_u$ depending on $u$; similarly for $X\sim_u Y$,
$X=O_u(Y)$, etc. We denote by $X_{\pm}$ any quantity of the form
$X\pm\varepsilon$ for any $\varepsilon>0$. Throughout this paper, we
denote $\frac{2d}{d-4}$ by $2^\#$.

For any spacetime slab $I\times\Bbb R^d$, we use
$L_t^qL_x^r(I\times\Bbb R^d)$ to denote the Banach space of
functions $u: I\times\Bbb R^d\rightarrow\Bbb C$ whose norm is
\begin{equation*}
\|u\|_{L_t^qL_x^r(I\times\Bbb R^d)}:=\big(\int_I
\|u(t)\|_{L_x^r}^q\big)^\frac{1}{q}<\infty,
\end{equation*}
with the usual modifications when $q$ or $r$ are equal to infinity.
When $q=r$ we abbreviate $L_t^qL_x^q$ as $L_{t,x}^q$.

We define the Fourier transform on $\Bbb R^d$ by
\begin{equation*}
\hat{f}(\xi):=(2\pi)^{-d/2}\int_{\Bbb R^d}e^{-ix\cdot\xi}f(x)dx.
\end{equation*}
For $s\in\Bbb R$, we define the fractional differentiation/integral
operator
\begin{equation*}
\widehat{|\nabla|^sf}(\xi):=|\xi|^s\hat{f}(\xi),
\end{equation*}
which in turn defines the homogeneous Sobolev norm
\begin{equation*}
\|f\|_{\dot{H}^s(\Bbb R^d)}:=\||\nabla|^sf\|_{L_x^2(\Bbb R^d)}.
\end{equation*}

We recall some basic facts about Littlewood-Paley theory. Let
$\varphi(\xi)$ be a radial bump function supported in the ball
$\{\xi\in\mathbb{R}^d: |\xi|\leq\frac{11}{10}\}$ and equal to 1 on
the ball $\{\xi\in\mathbb{R}^d: |\xi|\leq1\}$. For each number
$N>0$, we define the Fourier multipliers
\begin{align*}
\widehat{P_{\leq N}f}(\xi)&:=\varphi(\xi/N)\hat{f}(\xi),\\
\widehat{P_{\geq N}f}(\xi)&:=(1-\varphi(\xi/N))\hat{f}(\xi),\\
\widehat{P_{N}f}(\xi)&:=(\varphi(\xi/N)-\varphi(2\xi/N))\hat{f}(\xi)
\end{align*}
and similarly $P_{<N}$ and $P_{\geq N}$. We also define
$$P_{M<\cdot\leq N}:=P_{\leq N}-P_{\leq M}=\sum_{M<N'\leq N}P_{N'}$$
whenever $M<N$. We will usually use these multipliers when $M$ and
$N$ are dyadic numbers; in particular, all summations over $N$ or
$M$ are understood to be over dyadic numbers. Nevertheless, it will
occasionally be convenient to allow $M$ and $N$ to not be a power of
2. Note that $P_N$ is not truly a projection; to get around this, we
will occasionally need to use the fattened Littlewood-Paley
operators:
\begin{equation}
\tilde{P}_N:=P_{N/2}+P_N+P_{2N}.
\end{equation}
They obey $P_N\tilde{P}_N=\tilde{P}_NP_N=P_N$.

As all Fourier multipliers, the Littlewood-Paley operators commute
with the propagator $e^{it\Delta^2}$, as well as with the
differential operators such as $i\partial_t+\Delta^2$. We will use
the basic properties of these operators many times, including
\begin{lemma}[Bernstein estimates]\label{Bern} For $1\leq p\leq q\leq\infty$,
\begin{align*}\||\nabla|^{\pm s}P_Nf\|_{L_x^p(\mathbb{R}^d)}&\sim
N^{\pm s}\|P_Nf\|_{L_x^p(\mathbb{R}^d)},\\
\|P_{\leq N}f\|_{L_x^q(\mathbb{R}^d)}&\lesssim
N^{\frac{d}{p}-\frac{d}{q}}\|P_{\leq N}f\|_{L_x^p(\mathbb{R}^d)},\\
\|P_{N}f\|_{L_x^q(\mathbb{R}^d)}&\lesssim
N^{\frac{d}{p}-\frac{d}{q}}\|P_{N}f\|_{L_x^p(\mathbb{R}^d)}.
\end{align*}
\end{lemma}
We also need the following fractional chain rule \cite{CWe}:
\begin{lemma}[Fractional chain rule, \cite{CWe}]\label{fcr}
Suppose $G\in C^1(\Bbb C)$, $s\in(0, 1]$ and $1<p, p_1, p_2<\infty$
are such that $\frac{1}{p}=\frac{1}{p_1}+\frac{1}{p_2}$. Then,
\begin{equation*}
\big\||\nabla|^s
G(u)\big\|_p\lesssim\big\|G'(u)\big\|_{p_1}\big\||\nabla|^s
u\big\|_{p_2}.
\end{equation*}
\end{lemma}
Another tool we will use is a form of Gronwall's inequality that
involves both the past and the future, `acausal' in the terminology
of \cite{tao2}.
\begin{lemma}[A Gronwall inequality, \cite{kv}]\label{iter}
Given $\gamma>0$, $0<\eta<\frac{1}{2}(1-2^{-\gamma})$, and
$\{b_k\}\in \ell^\infty(\Bbb Z^+)$, let $x_k\in\ell^\infty(\Bbb
Z^+)$ be a non-negative sequence obeying
\begin{equation*}
x_k\leq b_k+\eta\sum_{l=0}^\infty2^{-\gamma|k-l|}x_l\ \ \text{for
all}\ \ k\geq0.
\end{equation*}
Then
\begin{equation*}
x_k\lesssim \sum_{l=0}^kr^{|k-l|}b_l\ \ \text{for all}\ \ k\geq0
\end{equation*}
for some $r=r(\eta)\in(2^{-\gamma},1)$. Moreover,
$r\downarrow2^{-\gamma}$ as $\eta\downarrow0$.
\end{lemma}
\section{Double Duhamel formula}
In this section, we prove the Double Duhamel formula. Similar
formula has appeared in \cite{TVZ}. For completeness, we give the
proof, see also \cite{TVZ}.
\begin{lemma}[Double Duhamel formula]\label{duh} Let $u$ be an almost periodic
solution to (\ref{fnls}) on its maximal-lifespan $I$. Then, for all
$t\in I$,
\begin{align*}
u(t)=&\lim_{T\nearrow\sup{I}}i\int_t^Te^{i(t-t')\Delta^2}F(u(t'))dt'\\
&=-\lim_{T\searrow\inf{I}}i\int_t^Te^{i(t-t')\Delta^2}F(u(t'))dt',
\end{align*}
as weak limits in $\dot{H}^2$.
\end{lemma}
\begin{proof}Suppose $u$ be an almost periodic
solution to (\ref{fnls}) on its maximal-lifespan $I$, then we claim
that $e^{-it\Delta^2}u(t)$ is weakly convergent in $\dot{H}^2(\Bbb
R^d)$ to zero as $t\rightarrow\sup{I}$ or $t\rightarrow\inf{I}$.

We just prove the claim as $t\rightarrow\sup{I}$, as the other case
is similar. By almost periodicity, we have a compact subset $K\in
\dot{H}_x^2(\Bbb R^d)$ and group elements $g_{\theta(t), x_0(t),
\lambda(t)}\in G$ for each $t\in I$ such that
\begin{equation}\label{equ21}
g_{\theta(t), x_0(t), \lambda(t)}^{-1}u(t)\in K.
\end{equation}
Suppose first that $\sup{I}$ is finite, and thus $u$ exhibits
forward blowup in finite time. By Corollary 4.10 in \cite{mxz3}, we
conclude that this forces $\lambda(t)$ to go to zero as
$t\rightarrow\sup{I}$. Thus the operator $g_{\theta(t), x_0(t),
\lambda(t)}$ are weakly convergent to zero. By the compactness of
$K$, this implies
\begin{equation*}
\lim_{t\rightarrow\sup{I}}\sup_{f\in K}\big|\langle \Delta
g_{\theta(t), x_0(t), \lambda(t)}f, \Delta\phi\rangle_{L_x^2(\Bbb
R^d)}\big|=0
\end{equation*}
for all $\phi\in\dot{H}^2(\Bbb R^d)$. From this and (\ref{equ21}),
we see that $u(t)$ converges weakly to zero as
$t\rightarrow\sup{I}$. Since $\sup{I}$ is finite and the propagator
curve $t\mapsto e^{-it\Delta^2}$ is continuous in the strong
operator topology, we see that $e^{-it\Delta^2}u(t)$ converges
weakly to zero, as desired.

Now suppose instead that $\sup{I}$ is infinite. It will suffice to
show that
\begin{equation*}
\lim_{t\rightarrow+\infty}\langle\Delta e^{-it\Delta^2}u(t),
\phi\rangle_{L_x^2(\Bbb R^d)}=0
\end{equation*}
for all test functions $\phi\in C_0^\infty(\Bbb R^d)$. Applying
(\ref{equ21}) and duality, it suffices to show that
\begin{equation*}
\lim_{t\rightarrow+\infty}\sup_{f\in K}\big|\langle\Delta
g_{\theta(t), x_0(t), \lambda(t)}f,
e^{it\Delta^2}\phi\rangle_{L_x^2(\Bbb R^d)}\big|=0.
\end{equation*}
By the compactness of $K$, it therefore suffices to show that
\begin{equation*}
\lim_{t\rightarrow+\infty}\big|\langle\Delta g_{\theta(t), x_0(t),
\lambda(t)}f, e^{it\Delta^2}\phi\rangle_{L_x^2(\Bbb R^d)}\big|=0
\end{equation*}
for each $f\in\dot{H}^2$. But the claim follows from the stationary
phase expansion of $e^{it\Delta^2}\phi$, the point being that
$e^{it\Delta^2}\phi$ acquires a quartic  phase oscillation as
$t\rightarrow\infty$ which cannot be renormalized by any of the
symmetries.

Now recall the Duhamel formula
\begin{equation*}
u(t)=e^{it\Delta^2}e^{-it_{+}\Delta^2}u(t_{+})+i\int_t^{t_+}e^{i(t-t')\Delta^2}F(u(t'))dt'
\end{equation*}
for any $t, t_+\in I$. Letting $t_+$ converge to $\sup{I}$, then we
conclude the backward Duhamel formula
\begin{equation*}
u(t)=i\int_t^{\sup{I}}e^{i(t-t')\Delta^2}F(u(t'))dt',
\end{equation*}
where the improper integral is interpreted in a conditionally
convergent sense in the weak topology, that is
\begin{equation*}
\langle u(t), f\rangle=\lim_{t_+\rightarrow\sup{I}}\big\langle\Delta
i\int_t^{t_+}e^{i(t-t')\Delta^2}F(u(t'))dt',
f\big\rangle_{L_x^2(\Bbb R^d)}
\end{equation*}
for all $f\in L_x^2(\Bbb R^d)$. Similarly, we have the forward
Duhamel formula
\begin{equation*}
u(t)=-i\int_{\inf{I}}^te^{i(t-t')\Delta^2}F(u(t'))dt'.
\end{equation*}
\end{proof}

\section{Finite time blow up}
In this section we preclude scenario I in Theorem \ref{enemies}. The
argument is essentially taken from \cite{kenig-merle}, see also
\cite{kv}, \cite{mxz3}.
\begin{theorem}[No finite-time blowup] Let $d\geq5$. Then there are
no maximal-lifespan solutions $u:I\times\Bbb R^d\rightarrow\Bbb C$
to (\ref{fnls}) that are almost periodic modulo $G$, obey
\begin{equation}\label{equ31}
S_I(u)=\infty
\end{equation}
and are such that either $|\inf I|<\infty$ or $\sup I<\infty$.
\end{theorem}
\begin{proof}
Suppose for a contradiction that there existed such a solution $u$.
Without loss of generality, we may assume $\sup I<\infty$. Then by
Corollary 4.10 in \cite{mxz3},
\begin{equation}\label{equ32}
\liminf_{t\nearrow\sup I}N(t)=\infty.
\end{equation}
We now show that this implies that
\begin{equation}\label{equ33}
\limsup_{t\nearrow\sup I}\int_{|x|\leq R}|u(t,x)|^2dx=0\ \ \text{for
all}\ \ R>0.
\end{equation}
Indeed, let $0<\eta<1$ and $t\in I$. By H\"older, Sobolev embedding
and energy conservation,
\begin{align*}
\int_{|x|\leq R}|u(t,x)|^2dx\lesssim& \int_{|x-x(t)|\leq\eta
R}|u(t,x)|^2dx+\int_{|x|\leq R\atop |x-x(t)|>\eta R}|u(t,x)|^2dx\\
\lesssim& (\eta R)^4\|u\|_{L^{2^{\#}}(\Bbb
R^d)}^2+R^4\Big(\int_{|x-x(t)|>\eta
R}|u(t,x)|^\frac{2d}{d-4}\Big)^\frac{d-4}{d}\\
\lesssim& (\eta R)^4 E(u)+R^4\Big(\int_{|x-x(t)|>\eta
R}|u(t,x)|^\frac{2d}{d-4}\Big)^\frac{d-4}{d}.
\end{align*}
Letting $\eta\rightarrow0$, we can make the first term on the
right-hand side of the inequality above as small as we wish. On the
other hand, by (\ref{equ32}), almost periodicity and Remark
\ref{rem1}, we see that
\begin{equation*}
\limsup_{t\nearrow\sup{I}}\int_{|x-x(t)|>\eta
R}|u(t,x)|^\frac{2d}{d-4}dx< \limsup_{t\nearrow\sup{I}}
\int_{|x-x(t)|>C(\epsilon)/N(t)}|u(t,x)|^\frac{2d}{d-4}dx<\epsilon,
\end{equation*}
for any $\epsilon>0$. Thus
\begin{equation*}
\limsup_{t\nearrow\sup{I}}\int_{|x-x(t)|>\eta
R}|u(t,x)|^\frac{2d}{d-4}dx=0
\end{equation*}
by the arbitrary of $\epsilon$.

For $t\in I$, define
\begin{equation*}
M_R(t):=\int_{\Bbb R^d}\phi\big(\frac{|x|}{R}\big)|u(t,x)|^2dx,
\end{equation*}
where $\phi$ is a smooth, radial function such that $\phi(r)=1$ for
$r\leq1$ and $\phi=0$ for $r\geq2$. By (\ref{equ33}),
\begin{equation}\label{equ84}
\limsup_{t\nearrow\sup I} M_R(t)=0\ \ \text{for all}\ \ R>0.
\end{equation}
On the other hand,
\begin{equation*}
\partial_tM_R(t)=\ -2{\rm
Im}\int\Delta\Big(\phi\big(\frac{|x|}{R}\big)\Big)\bar{u}\Delta
udx-2{\rm
Im}\int\nabla\Big(\phi\big(\frac{|x|}{R}\big)\Big)\cdot\nabla\bar{u}\Delta
udx.
\end{equation*}
So by H\"older and Hardy's inequality, we have
\begin{align*}
|\partial_tM_R(t)|\lesssim&\ \int_{|x|\sim R}\frac{|u||\Delta
u|}{R^2}dx+\int_{|x|\sim R}\frac{|\nabla u||\Delta u|}{R}\\
\lesssim&\ \big\|\frac{u}{|x|^2}\big\|_2\|\Delta
u\|_2+\big\|\frac{|\nabla u|}{|x|}\big\|_2\|\Delta u\|_2\\
\lesssim&\ E(u).
\end{align*}
Thus,
\begin{equation*}
M_R(t_1)=M_R(t_2)+\int_{t_2}^{t_1}\partial_tM_R(t)dt\lesssim
M_R(t_2)+|t_1-t_2|E(u)
\end{equation*}
for all $t_1, t_2\in I$ and $R>0$. Let $t_2\nearrow\sup I$ and
invoking (\ref{equ84}), we have
\begin{equation*}
M_R(t_1)\lesssim |\sup I-t_1|E(u).
\end{equation*}
Now letting $R\rightarrow\infty$ and using the conservation of mass,
we obtain $u_0\in L_x^2(\Bbb R^d)$. Finally, letting
$t_1\nearrow\sup I$, we deduce $u_0=0$. Thus $u\equiv0$, which
contradicts (\ref{equ31}).
\end{proof}
\section{Negative regularity}
\begin{theorem}[Negative regularity in global case]\label{reg} Let
$d\geq9$ and let $u$ be a global solution to (\ref{fnls}) that is
almost periodic modulo $G$. Suppose also that $E(u)<\infty$ and
\begin{equation}\label{equ41}
\inf_{t\in \Bbb R}N(t)\geq1.
\end{equation}
Then $u\in L_t^\infty\dot{H}^{-\epsilon}(\Bbb R\times\Bbb R^d)$ for
some $\epsilon=\epsilon(d)>0$. In particular, $u\in L_t^\infty
L_x^2(\Bbb R\times\Bbb R^d)$.
\end{theorem}
Let $u$ be a solution to (\ref{fnls}) that obeys the hypothesis of
Theorem \ref{reg}. Let $\eta>0$ be a small constant to be chosen
later. Then by Remark \ref{rem1} combined with (\ref{equ41}), there
exists $N_0=N_0(\eta)$ such that
\begin{equation}\label{equ42}
\big\|\Delta u_{\leq N_0}\big\|_{L_t^\infty L_x^2(\Bbb R\times\Bbb
R^d)}\leq\eta.
\end{equation}
We define
\begin{equation*}\label{an} A(N)=\left\{ \aligned
    & N^{-\frac{4}{d-4}}\sup_{t\in\Bbb R}\|u_N(t)\|_{L_x^\frac{2d(d-4)}{d^2-8d+8}(\Bbb R^d)}\ \ \text{for}\ d\geq12\\
     & N^{-\frac{1}{2}}\sup_{t\in\Bbb R}\|u_N(t)\|_{L_x^\frac{2d}{d-5}(\Bbb R^d)}\ \ \ \ \ \ \ \ \ \text{for}\ 9\leq d<12
\endaligned
\right.
\end{equation*}
for frequencies $N<10N_0$.

We next prove a recurrence formula for $A(N)$.
\begin{lemma}\label{recur}
For all $N<10N_0$,
\begin{equation*}
A(N)\lesssim_u\
\big(\frac{N}{N_0}\big)^\alpha+\eta^\frac{8}{d-4}\sum_{\frac{N}{10}\leq
N_1\leq N_0}\big(\frac{N}{N_1}\big)^\alpha
A(N_1)+\eta^\frac{8}{d-4}\sum_{N_1<\frac{N}{10}}\big(\frac{N_1}{N}\big)^\alpha
A(N_1),
\end{equation*}
where $\alpha=\min\{\frac{4}{d-4}, \frac{1}{2}\}$.
\end{lemma}
\begin{proof} We first give the proof in dimensions $d\geq12$. Fix
$N\leq 10N_0$, by time translation symmetry, it suffices to prove
\begin{equation*}
N^{-\frac{4}{d-4}}\|u_N(0)\|_{L_x^\frac{2d(d-4)}{d^2-8d+8}(\Bbb
R^d)}\lesssim_u
\big(\frac{N}{N_0}\big)^\frac{4}{d-4}+\eta^\frac{8}{d-4}\sum_{\frac{N}{10}\leq
N_1\leq N_0}\big(\frac{N}{N_1}\big)^\frac{4}{d-4}
A(N_1)+\eta^\frac{8}{d-4}\sum_{N_1<\frac{N}{10}}\big(\frac{N_1}{N}\big)^\frac{4}{d-4}
A(N_1).
\end{equation*}
Using Lemma \ref{duh} into the future followed by the triangle
inequality, Bernstein and the dispersive estimate, that is (3.7) in
\cite{Pa}, we estimate
\begin{align}
&N^{-\frac{4}{d-4}}\|u_N(0)\|_{L_x^\frac{2d(d-4)}{d^2-8d+8}(\Bbb
R^d)}\nonumber\\
\lesssim&\
N^{-\frac{4}{d-4}}\Big\|\int_0^{N^{-4}}e^{-it\Delta^2}P_NF(u(t))dt\Big\|_{L_x^\frac{2d(d-4)}{d^2-8d+8}(\Bbb
R^d)}+N^{-\frac{4}{d-4}}\Big\|\int_{N^{-4}}^\infty
e^{-it\Delta^2}P_NF(u(t))dt\Big\|_{L_x^\frac{2d(d-4)}{d^2-8d+8}(\Bbb
R^d)}\nonumber\\
\lesssim&\
N^{2}\Big\|\int_0^{N^{-4}}e^{-it\Delta^2}P_NF(u(t))dt\Big\|_{L_x^2(\Bbb
R^d)}+N^{-\frac{4}{d-4}}\int_{N^{-4}}^\infty
t^{-\frac{d-2}{d-4}}dt\big\|P_NF(u)\big\|_{L_t^\infty
L_x^\frac{2d(d-4)}{d^2-8}(\Bbb R\times\Bbb R^d)}\nonumber\\
\lesssim&\ N^{-2}\|P_N(F(u))\|_{L_t^\infty L_x^2(\Bbb
R^d)}+N^\frac{4}{d-4}\big\|P_NF(u)\big\|_{L_t^\infty
L_x^\frac{2d(d-4)}{d^2-8}(\Bbb R\times\Bbb R^d)}\nonumber\\
\lesssim&\ N^\frac{4}{d-4}\big\|P_NF(u)\big\|_{L_t^\infty
L_x^\frac{2d(d-4)}{d^2-8}(\Bbb R\times\Bbb R^d)}.\label{equ43}
\end{align}
Using the Fundamental Theorem of Calculus, we decompose
\begin{align}\label{equ44}
F(u)=&\ O(|u_{>N_0}||u_{\leq
N_0}|^\frac{8}{d-4})+O(|u_{>N_0}|^\frac{d+4}{d-4})+F(u_{\frac{N}{10}\leq\cdot\leq
N_0})\nonumber\\
&\ +u_{<\frac{N}{10}}\int_0^1F_z(u_{\frac{N}{10}\leq\cdot\leq
N_0}+\theta u_{<\frac{N}{10}})d\theta\\
&\
+\overline{u_{<\frac{N}{10}}}\int_0^1F_{\bar{z}}(u_{\frac{N}{10}\leq\cdot\leq
N_0}+\theta u_{<\frac{N}{10}})d\theta.\nonumber
\end{align}
The contribution to the right-hand side of (\ref{equ43}) coming from
terms that contain at least one copy of $u_{>N_0}$ can be estimated
in the following manner: Using H\"older, Bernstein and the energy
conservation, we have
\begin{align*}
&\ N^\frac{4}{d-4}\big\|P_NO(|u_{>N_0}||u_{\leq
N_0}|^\frac{8}{d-4})\big\|_{L_t^\infty
L_x^\frac{2d(d-4)}{d^2-8}(\Bbb R\times\Bbb R^d)}\\
\lesssim&\ N^\frac{4}{d-4}\big\|u_{>N_0}\big\|_{L_t^\infty
L_x^\frac{2d(d-4)}{d^2-8d+24}(\Bbb R\times\Bbb
R^d)}\big\|u\big\|_{L_t^\infty L_x^\frac{2d}{d-4}(\Bbb R\times\Bbb
R^d)}^\frac{8}{d-4}\\
\lesssim&\
\big(\frac{N}{N_0}\big)^\frac{4}{d-4}\big\||\nabla|^{\frac4{d-4}}u_{>N_0}\big\|_{L_t^\infty
L_x^\frac{2d(d-4)}{d^2-8d+24} (\Bbb R\times\Bbb
R^d)}\big\|u\big\|_{L_t^\infty
L_x^\frac{2d}{d-4}(\Bbb R\times\Bbb R^d)}^\frac{8}{d-4}\\
\lesssim&\ \big(\frac{N}{N_0}\big)^\frac{4}{d-4}.
\end{align*}
Next we turn to the contribution to the right-hand side of
(\ref{equ43}) coming from the last two terms in (\ref{equ44}). It
suffices to consider the first of them since similar arguments can
be used to deal with the second.

First we note that $\Delta u\in L_t^\infty L_x^2$, we have
\begin{equation*}
F_z(u)\in \dot{\Lambda}^{\frac{d(d-4)}{4(d-2)},
\infty}_{\frac{8}{d-4}}.
\end{equation*}
Furthermore, as $P_{>\frac{N}{10}}F_z(u)$ is restricted to high
frequencies, the Besov characterization of the homogeneous H\"older
continuous functions (see \cite{St}, $\S$VI. 7.8) yields
\begin{equation*}
\big\|P_{>\frac{N}{10}}F_z(u)\big\|_{L_t^\infty
L_x^\frac{d(d-4)}{4(d-2)}(\Bbb R\times\Bbb R^d)}\lesssim\
N^{-\frac{8}{d-4}}\|\Delta u\|_{L_t^\infty L_x^2(\Bbb R\times\Bbb
R^d)}^\frac{8}{d-4}.
\end{equation*}
In fact,
\begin{align*}
\big\|P_{>\frac{N}{10}}F_z(u)\big\|_{L_t^\infty
L_x^\frac{d(d-4)}{4(d-2)}(\Bbb R\times\Bbb R^d)}\lesssim&\
\sum_{M>\frac{N}{10}}\big\|P_MF_z(u)\big\|_{L_t^\infty
L_x^\frac{d(d-4)}{4(d-2)}(\Bbb R\times\Bbb R^d)}\\
\lesssim&\ \sum_{M>\frac{N}{10}}M^{-\frac{8}{d-4}}\|\Delta
u\|_{L_t^\infty L_x^2(\Bbb R\times\Bbb R^d)}^\frac{8}{d-4}\\
\lesssim&\ N^{-\frac{8}{d-4}}\|\Delta u\|_{L_t^\infty L_x^2(\Bbb
R\times\Bbb R^d)}^\frac{8}{d-4}.
\end{align*}
Thus, by H\"older and (\ref{equ42}),
\begin{align*}
&N^\frac{4}{d-4}\big\|P_N\big(u_{<\frac{N}{10}}\int_0^1F_z(u_{\frac{N}{10}\leq\cdot\leq
N_0}+\theta u_{<\frac{N}{10}})d\theta\big)\big\|_{L_t^\infty
L_x^\frac{2d(d-4)}{d^2-8}(\Bbb R\times\Bbb R^d)}\\
\lesssim&\ N^\frac{4}{d-4}\big\|u_{<\frac{N}{10}}\big\|_{L_t^\infty
L_x^\frac{2d(d-4)}{d^2-8d+8}(\Bbb R\times\Bbb
R^d)}\big\|P_{>\frac{N}{10}}\big(\int_0^1F_z(u_{\frac{N}{10}\leq\cdot\leq
N_0}+\theta u_{<\frac{N}{10}})d\theta\big)\big\|_{L_t^\infty
L_x^\frac{d(d-4)}{4(d-2)}(\Bbb R\times\Bbb R^d)}\\
\lesssim&\ N^{-\frac{4}{d-4}}\|u_{<\frac{N}{10}}\|_{L_t^\infty
L_x^\frac{2d(d-4)}{d^2-8d+8}(\Bbb R\times\Bbb R^d)}\|\Delta
u_{<N_0}\|_{L_t^\infty
L_x^2(\Bbb R\times\Bbb R^d)}^\frac{8}{d-4}\\
\lesssim&\
\eta^\frac{8}{d-4}\sum_{N_1<\frac{N}{10}}\big(\frac{N_1}{N}\big)^\frac{4}{d-4}A(N_1).
\end{align*}
Hence, the contribution coming from the last two terms in
(\ref{equ43}) is acceptable.

We are left to estimate the contribution of
$F(u_{\frac{N}{10}\leq\cdot\leq N_0})$. We need only show
\begin{equation}\label{equ45}
\big\|F(u_{\frac{N}{10}\leq\cdot\leq N_0})\big\|_{L_t^\infty
L_x^\frac{2d(d-4)}{d^2-8}(\Bbb R\times\Bbb R^d)}\lesssim
\eta^\frac{8}{d-4}\sum_{\frac{N}{10}\leq N_1\leq
N_0}N_1^{-\frac{4}{d-4}}A(N_1).
\end{equation}
As $d\geq12$, we have $\frac{8}{d-4}\leq1$. Using the triangle
inequality, Bernstein, (\ref{equ42}) and H\"older, we estimate
\begin{align*}
&\quad\big\|F(u_{\frac{N}{10}\leq\cdot\leq N_0})\big\|_{L_t^\infty
L_x^\frac{2d(d-4)}{d^2-8}(\Bbb R\times\Bbb R^d)}\\
\lesssim&\ \sum_{\frac{N}{10}\leq N_1\leq
N_0}\big\|u_{N_1}|u_{\frac{N}{10}\leq\cdot\leq
N_0}|^\frac{8}{d-4}\big\|_{L_t^\infty L_x^\frac{2d(d-4)}{d^2-8}(\Bbb
R\times\Bbb R^d)}\\
\lesssim&\ \sum_{\frac{N}{10}\leq N_1, N_2\leq
N_0}\big\|u_{N_1}|u_{N_2}|^\frac{8}{d-4}\big\|_{L_t^\infty
L_x^\frac{2d(d-4)}{d^2-8}(\Bbb R\times\Bbb R^d)}\\
\lesssim&\ \sum_{\frac{N}{10}\leq N_1\leq N_2\leq
N_0}\big\|u_{N_1}\big\|_{L_t^\infty
L_x^\frac{2d(d-4)}{d^2-8d+8}(\Bbb R\times\Bbb
R^d)}\big\|u_{N_2}\big\|_{L_t^\infty L_x^\frac{2d}{d-2}}^\frac{8}{d-4}\\
&\quad\quad+\sum_{\frac{N}{10}\leq N_2\leq N_1\leq
N_0}\big\|u_{N_1}\big\|_{L_t^\infty
L_x^\frac{2d(d-4)^2}{d^3-12d^2+56d-32}(\Bbb R\times\Bbb
R^d)}\big\|u_{N_2}\big\|_{L_t^\infty
L_x^\frac{2d(d-4)}{d^2-8d+8}(\Bbb R\times\Bbb R^d)}^\frac{8}{d-4}\\
\lesssim&\ \sum_{\frac{N}{10}\leq N_1\leq N_2\leq
N_0}\eta^\frac{8}{d-4}N_2^{-\frac{8}{d-4}}\big\|u_{N_1}\big\|_{L_t^\infty
L_x^\frac{2d(d-4)}{d^2-8d+8}(\Bbb R\times\Bbb R^d)}\\
&\quad+\sum_{\frac{N}{10}\leq N_2\leq N_1\leq
N_0}\big\|u_{N_1}\big\|_{L_t^\infty L_x^\frac{2d}{d-2}(\Bbb
R\times\Bbb R^d)}^\frac{8}{d-4}\big\|u_{N_1}\big\|_{L_t^\infty
L_x^\frac{2d(d-4)}{d^2-8d+8}(\Bbb R\times\Bbb
R^d)}^\frac{d-12}{d-4}\big\|u_{N_2}\big\|_{L_t^\infty
L_x^\frac{2d(d-4)}{d^2-8d+8}(\Bbb R\times\Bbb R^d)}^\frac{8}{d-4}\\
\lesssim&\ \eta^\frac{8}{d-4}\sum_{\frac{N}{10}\leq N_1\leq
N_0}N_1^{-\frac{4}{d-4}}A(N_1)\\
&\quad+\eta^\frac{8}{d-4}\sum_{\frac{N}{10}\leq N_2\leq N_1\leq
N_0}N_1^{-\frac{8}{d-4}}\big\|u_{N_1}\big\|_{L_t^\infty
L_x^\frac{2d(d-4)}{d^2-8d+8}(\Bbb R\times\Bbb
R^d)}^\frac{d-12}{d-4}\big\|u_{N_2}\big\|_{L_t^\infty
L_x^\frac{2d(d-4)}{d^2-8d+8}(\Bbb R\times\Bbb R^d)}^\frac{8}{d-4}\\
\lesssim&\ \eta^\frac{8}{d-4}\sum_{\frac{N}{10}\leq N_1\leq
N_0}N_1^{-\frac{4}{d-4}}A(N_1)\\
& +\eta^\frac{8}{d-4}\sum_{\frac{N}{10}\leq N_2\leq N_1\leq N_0}
\left(\frac{N_2}{N_1}\right)^{\frac{64}{(d-4)^2}}
\left(N^{-\frac4{d-4}}_1A(N_1)\right)^{\frac{d-12}{d-4}}
\left(N^{-\frac4{d-4}}_2A(N_2)\right)^{\frac8{d-4}}\\
\lesssim&\ \eta^\frac{8}{d-4}\sum_{\frac{N}{10}\leq N_1\leq
N_0}N_1^{-\frac{4}{d-4}}A(N_1).
\end{align*}
This proves (\ref{equ45}) and so completes the proof of the lemma in
dimensions $d\geq12$.

Consider now $9\leq d<12$. Arguing as for (\ref{equ43}), we have
\begin{equation*}
N^{-\frac{1}{2}}\big\|u_N(0)\big\|_{L_x^\frac{2d}{d-5}(\Bbb
R^d)}\lesssim N^\frac{1}{2}\big\|P_NF(u)\big\|_{L_t^\infty
L_x^\frac{2d}{d+5}(\Bbb R\times\Bbb R^d)},
\end{equation*}
which we estimate by decomposing the nonlinearity as in
(\ref{equ44}). First we have \begin{align*}
&N^\frac{1}{2}\big\|P_NO(|u_{>N_0}||u_{\leq
N_0}|^\frac{8}{d-4})\big\|_{L_t^\infty L_x^\frac{2d}{d+5}(\Bbb
R\times\Bbb R^d)}\\
\lesssim&\ N^\frac{1}{2}\big\|u_{>N_0}\big\|_{L_t^\infty
L_x^\frac{2d}{d-3}(\Bbb R\times\Bbb R^d)}\big\|u\big\|_{L_t^\infty
L_x^\frac{2d}{d-4}(\Bbb R\times\Bbb R^d)}^\frac{8}{d-4}\\
\lesssim&\ \big(\frac{N}{N_0}\big)^\frac{1}{2}.
\end{align*}
Next using Bernstein and Lemma \ref{fcr} together with
(\ref{equ42}), we have
\begin{align*}
&N^\frac{1}{2}\big\|P_N\big(u_{<\frac{N}{10}}\int_0^1F_z(u_{\frac{N}{10}\leq\cdot\leq
N_0}+\theta u_{<\frac{N}{10}})d\theta\big)\big\|_{L_t^\infty
L_x^\frac{2d}{d+5}(\Bbb R\times\Bbb R^d)}\\
\lesssim&\ N^\frac{1}{2}\big\|u_{<\frac{N}{10}}\big\|_{L_t^\infty
L_x^\frac{2d}{d-5}}\big\|P_{>\frac{N}{10}}F_z(u_{\frac{N}{10}\leq\cdot\leq
N_0}+\theta u_{<\frac{N}{10}})\big\|_{L_t^\infty
L_x^\frac{d}{5}(\Bbb R\times\Bbb R^d)}\\
\lesssim&\ N^{-\frac{1}{2}}\big\|u_{<\frac{N}{10}}\big\|_{L_t^\infty
L_x^\frac{2d}{d-5}}\big\|\nabla F_z(u_{\frac{N}{10}\leq\cdot\leq
N_0}+\theta u_{<\frac{N}{10}})\big\|_{L_t^\infty
L_x^\frac{d}{5}(\Bbb R\times\Bbb R^d)}\\
\lesssim&\ N^{-\frac{1}{2}}\big\|u_{<\frac{N}{10}}\big\|_{L_t^\infty
L_x^\frac{2d}{d-5}}\big\|\Delta u_{<N_0}\big\|_{L_t^\infty
L_x^2}^\frac{8}{d-4}\\
\lesssim&\
\eta^\frac{8}{d-4}\sum_{N_1<\frac{N}{10}}\big(\frac{N_1}{N}\big)^\frac{1}{2}A(N_1).
\end{align*}
Finally we estimate
\begin{equation*}
\big\|F(u_{\frac{N}{10}\leq\cdot\leq N_0})\big\|_{L_t^\infty
L_x^\frac{2d}{d+5}(\Bbb R\times\Bbb R^d)}.
\end{equation*}
We denote the maximal integer less than or equal to
$\frac{d+4}{d-4}$ by $k(d)$, then
\begin{align*}
&\big\|F(u_{\frac{N}{10}\leq\cdot\leq N_0})\big\|_{L_t^\infty
L_x^\frac{2d}{d+5}(\Bbb R\times\Bbb R^d)}\\
\lesssim&\ \sum_{\frac{N}{10}\leq N_1, \cdots, N_{k(d)}, M\leq
N_0}\big\|u_{N_1}u_{N_2}\cdots
u_{N_{k(d)}}|u_M|^{\frac{d+4}{d-4}-k(d)}\big\|_{L_t^\infty
L_x^\frac{2d}{d+5}(\Bbb R\times\Bbb R^d)}\\
\lesssim&\ \sum_{\frac{N}{10}\leq N_1, \cdots, N_{k(d)}, M\leq N_0
\atop N_1=\min\{N_1,\cdots, N_{k(d)},
M\}}\big\|u_{N_1}\big\|_{L_t^\infty L_x^\frac{2d}{d-5}(\Bbb
R\times\Bbb R^d)}\prod_{j=2}^{k(d)}\big\|u_{N_j}\big\|_{L_t^\infty
L_x^\frac{8d}{5(d-4)}(\Bbb R\times\Bbb
R^d)}\big\|u_M\big\|_{L_t^\infty L_x^\frac{8d}{5(d-4)}(\Bbb
R\times\Bbb R^d)}^{\frac{d+4}{d-4}-k(d)}\\
&\ \quad+\sum_{\frac{N}{10}\leq N_1, \cdots, N_{k(d)}, M\leq N_0
\atop M=\min\{N_1,\cdots, N_{k(d)}, M\}}\big\|u_M\big\|_{L_t^\infty
L_x^\frac{2d}{d-5}(\Bbb R\times\Bbb
R^d)}^{\frac{d+4}{d-4}-k(d)}\big\|u_{N_1}\big\|_{L_t^\infty
L_x^\frac{2d}{d-5}(\Bbb R\times\Bbb
R^d)}^{1+k(d)-\frac{d+4}{d-4}}\\
&\quad\quad\quad\quad\quad\quad\quad\quad \quad\quad\quad\quad
\quad\quad\quad\quad \times\big\|u_{N_1}\big\|_{L_t^\infty
L_x^\frac{8d}{5(d-4)}(\Bbb R\times\Bbb
R^d)}^{\frac{d+4}{d-4}-k(d)}\prod_{j=2}^{k(d)}\big\|u_{N_j}\big\|_{L_t^\infty
L_x^\frac{8d}{5(d-4)}(\Bbb R\times\Bbb R^d)}\\
\end{align*}
\begin{align*}
\lesssim&\ \eta^\frac{8}{d-4}\sum_{\frac{N}{10}\leq N_1\leq
N_0}N_1^{-\frac{1}{2}}A(N_1)\\
&\ \quad+\sum_{\frac{N}{10}\leq M\leq N_1\leq
N_0}\big(\frac{N_1}{M}\big)^{k(d)+1-\frac{d+4}{d-4}}\big(M^{-\frac{1}{2}}\big\|u_{M}\big\|_{L_t^\infty
L_x^\frac{2d}{d-5}(\Bbb R\times\Bbb
R^d)}\big)^{\frac{d+4}{d-4}-k(d)}\\
&\quad\quad\quad\quad\times\big(N_1^{-\frac{1}{2}}\big\|u_{N_1}\big\|_{L_t^\infty
L_x^\frac{2d}{d-5}(\Bbb R\times\Bbb
R^d)}\big)^{1+k(d)-\frac{d+4}{d-4}}\\
\lesssim&\ \eta^\frac{8}{d-4}\sum_{\frac{N}{10}\leq N_1\leq
N_0}N_1^{-\frac{1}{2}}A(N_1).
\end{align*}
Putting everything together completes the proof of the lemma in the
case of $9\leq d<12$.
\end{proof}
\begin{proposition}[$L^p$ breach of scaling]\label{breach} Let $u$ be as in
Theorem \ref{reg}. Then
\begin{equation*}
u\in L_t^\infty L_x^p \ \ \text{for}\ \ \frac{2d(d+4)}{d^2-8}\leq
p<\frac{2d}{d-4}.
\end{equation*}
In particular, \begin{equation}\label{equ46} F(u)\in L_t^\infty
\dot{\Lambda}_{2}^{r,\infty}\ \ \text{for}\ \
\frac{2d(d+4)(d-4)}{d^3+8d^2-16d-64}\leq r<\frac{2d}{d+8}.
\end{equation}
\end{proposition}
\begin{proof} We only present the details for $d\geq12$. The
treatment of $9\leq d<12$ is completely analogous. Combining Lemma
\ref{recur} and Lemma \ref{iter}, we deduce
\begin{equation}\label{equ47}
\big\|u_N\big\|_{L_t^\infty L_x^\frac{2d(d-4)}{d^2-8d+8}(\Bbb
R\times\Bbb R^d)}\lesssim_u N^{\frac{8}{d-4}-}\ \ \text{for all}\ \
N\leq 10N_0
\end{equation}
In applying Lemma \ref{iter}, we set $N=10\cdot 2^{-k}N_0$,
$x_k=A(10\cdot 2^{-k}N_0)$ and take $\eta$ sufficiently small.

By interpolation followed by (\ref{equ47}), Bernstein and energy
conservation,
\begin{align*}
\big\|u_N\big\|_{L_t^\infty L_x^p}\lesssim&\
\big\|u_N\big\|_{L_t^\infty
L_x^\frac{2d(d-4)}{d^2-8d+8}}^{(d-4)(\frac{d-2}{2d}-\frac{1}{p})}\big\|u_N\big\|_{L_t^\infty
L_x^\frac{2d}{d-2}}^{1-(d-4)(\frac{d-2}{2d}-\frac{1}{p})}\\
\lesssim&\
N^{8(\frac{d-2}{2d}-\frac{1}{p})-}N^{-1+(d-4)(\frac{d-2}{2d}-\frac{1}{p})}\\
\lesssim&_u N^{\frac{d}{2}-\frac{4}{d}-\frac{d+4}{p}-}
\end{align*}
for all $N\leq 10N_0$. Then using Bernstein, we have
\begin{align*}
\big\|u\big\|_{L_t^\infty L_x^p}\lesssim&\ \big\|u_{\leq
N_0}\big\|_{L_t^\infty L_x^p}+\big\|u_{>N_0}\big\|_{L_t^\infty
L_x^p}\\
\lesssim&_u \sum_{N\leq
N_0}N^{\frac{d}{2}-\frac{4}{d}-\frac{d+4}{p}-}+\sum_{N>N_0}N^{\frac{d-4}{2}-\frac{d}{p}}\\
\lesssim&_u 1.
\end{align*}
(\ref{equ46}) follows by paraproduct and H\"oder inequality.
\end{proof}
\begin{proposition}[Some negative regularity]\label{neg} Let $d\geq9$ and let
$u$ be as in Theorem \ref{reg}. Assume further that $F(u)\in
L_t^\infty \dot{\Lambda}_s^{r,\infty}$ for some
$\frac{2d(d+4)(d-4)}{d^3+8d^2-16d-64}\leq r<\frac{2d}{d+8}$ and some
$0\leq s\leq 2$. Then there exists $s_0=s_0(r,d)>0$ such that $u\in
L_t^\infty \dot{H}^{s-s_0+}$.
\end{proposition}
\begin{proof} It suffices to prove that
\begin{equation}\label{equ48}
\big\||\nabla|^s u_N\big\|_{L_t^\infty L_x^2}\lesssim_u N^{s_0}\ \
\text{for all}\ \ N>0\ \text{and}\ \
s_0=\frac{d}{r}-\frac{d}{2}-4>0.
\end{equation}
Indeed, by Bernstein combined with energy conservation,
\begin{align*}
\big\||\nabla|^{s-s_0+} u_N\big\|_{L_t^\infty L_x^2}\leq&\
\big\||\nabla|^{s-s_0+} u_{\leq1}\big\|_{L_t^\infty
L_x^2}+\big\||\nabla|^{s-s_0+} u_{>1}\big\|_{L_t^\infty L_x^2}\\
\lesssim&_u\sum_{N\leq1}N^{0+}+\sum_{N>1}N^{(s-s_0+)-2}\\
\lesssim&_u1.
\end{align*}
We are left to prove (\ref{equ48}). By time-translation symmetry, it
suffices to prove
\begin{equation*}
\big\||\nabla|^s u_N(0)\big\|_{L_x^2}\lesssim_u N^{s_0}\ \ \text{for
all}\ N>0 \ \ \text{and}\ \ s_0=\frac{d}{r}-\frac{d+8}{2}>0.
\end{equation*}
Using the Duhamel formula (\ref{duh}) both in the future and in the
past, we write
\begin{align*}
&\ \big\||\nabla|^s u_N(0)\big\|_{L_x^2}^2\\
=&\ \lim_{T\rightarrow\infty}\lim_{T'\rightarrow-\infty}\Big\langle
i\int_0^Te^{-it\Delta^2}P_N|\nabla|^sF(u(t))dt,
-i\int_{T'}^0e^{-i\tau\Delta^2}P_N|\nabla|^sF(u(\tau))d\tau\Big\rangle\\
\leq&\int_0^{+\infty}\int_{-\infty}^0\big|\langle
P_N|\nabla|^sF(u(t)),
e^{i(t-\tau)\Delta^2}P_N|\nabla|^sF(u(\tau))\rangle\big|dtd\tau.
\end{align*}
We estimate the term inside the integrals in two ways. On one hand,
using H\"older and the dispersive estimate, \begin{align*}
&\big|\langle P_N|\nabla|^sF(u(t)),
e^{i(t-\tau)\Delta^2}P_N|\nabla|^sF(u(\tau))\rangle\big|\\
\lesssim&\
\big\|P_N|\nabla|^sF(u(t))\big\|_{L_x^r}\big\|e^{i(t-\tau)\Delta^2}P_N|\nabla|^sF(u(\tau))\big\|_{L_x^{r'}}\\
\lesssim&|t-\tau|^{-\frac{d}{2}(\frac{1}{r}-\frac{1}{2})}\big\|F(u(t))\big\|_{\dot{\Lambda}_s^{r,\infty}}^2.
\end{align*}
On the other hand, using Bernstein,
\begin{align*}
&\big|\langle P_N|\nabla|^sF(u(t)),
e^{i(t-\tau)\Delta^2}P_N|\nabla|^sF(u(\tau))\rangle\big|\\
\lesssim&\
\big\|P_N|\nabla|^sF(u(t))\big\|_{L_x^2}\big\|e^{i(t-\tau)\Delta^2}P_N|\nabla|^sF(u(\tau))\big\|_{L_x^2}\\
\lesssim&\
N^{2d(\frac{1}{r}-\frac{1}{2})}\big\|F(u(t))\big\|_{\dot{\Lambda}_s^{r,\infty}}^2.
\end{align*}
Thus,
\begin{align*}
\big\||\nabla|^s u_N(0)\big\|^2_{L_x^2}\lesssim&
\big\|F(u(t))\big\|_{\dot{\Lambda}_s^{r,\infty}}^2\int_0^\infty\int_{-\infty}^0
\min\{|t-\tau|^{-\frac{d}{2}(\frac{1}{r}-\frac{1}{2})},
N^{-2d(\frac{1}{2}-\frac{1}{r})}\}dtd\tau\\
\lesssim&
N^{2s_0}\big\|F(u(t))\big\|_{\dot{\Lambda}_s^{r,\infty}}^2,
\end{align*}
where we use the fact that $r<\frac{2d}{d+8}$. It's here that the
dimension restriction is imposed.
\end{proof}

{\it Proof of Theorem \ref{reg}.} Proposition \ref{breach} allows us
to apply Proposition \ref{neg} with $s=2$. We conclude that $u\in
L_t^\infty\dot{H}^{2-s_0+}$ for some $s_0=s_0(r,d)>0$. Thus we
deduce that $F(u)\in L_t^\infty \dot{\Lambda}_{2-s_0}^{r,\infty}$
for some $\frac{2d(d+4)(d-4)}{d^3+8d^2-16d-64}\leq
r<\frac{2d}{d+8}$. We are thus in the position to apply Proposition
\ref{neg} again and obtain $u\in L_t^\infty \dot{H}^{2-2s_0+}$.
Iterating this procedure finitely many times, we derive $u\in
L_t^\infty\dot{H}^{-\varepsilon}$ for some $0<\varepsilon<s_0$.

This completes the proof of Theorem \ref{reg}.
\section{The low-to-high frequency cascade}
In this section, we use the negative regularity provided by Theorem
\ref{reg} to preclude low-to-high frequency cascade solutions.
\begin{theorem}[Absence of cascades]\label{lth} Let $d\geq9$. There are no
global solutions to (\ref{fnls}) that are low-to-high frequency
cascades in the sense of Theorem \ref{enemies}.
\end{theorem}
\begin{proof}
Suppose for a contradiction that there existed such a solution $u$.
Then by Theorem \ref{reg}, $u\in L_t^\infty L_x^2$. Thus by the
conservation of mass,
\begin{equation*}
0\leq M(u)=M(u(t))=\int_{\Bbb R^d}|u(t,x)|^2dx<\infty\ \ \text{for
all}\ \ t\in\Bbb R.
\end{equation*}
Fix $t\in\Bbb R$ and let $\eta>0$ be a small constant. By
compactness, \begin{equation*} \int_{|\xi|\leq
c(\eta)N(t)}|\xi|^4|\hat{u}(t,\xi)|^2d\xi\leq\eta.
\end{equation*}
On the other hand, as $u\in L_t^\infty\dot{H}^{-\varepsilon}$ for
some $\varepsilon>0$,
\begin{equation*}
\int_{|\xi|\leq
c(\eta)N(t)}|\xi|^{-2\varepsilon}|\hat{u}(t,\xi)|^2d\xi\lesssim_u1.
\end{equation*}
Hence, by H\"older,
\begin{equation*}
\int_{|\xi|\leq c(\eta)N(t)}|\hat{u}(t,\xi)|^2d\xi\lesssim_u
\eta^\frac{\varepsilon}{2+\varepsilon}.
\end{equation*}
Meanwhile,
\begin{align*}
\int_{|\xi|\geq c(\eta)N(t)}|\hat{u}(t,\xi)|^2d\xi\leq&
[c(\eta)N(t)]^{-4}\int_{\Bbb R^d}|\xi|^4|\hat{u}(t,\xi)|^2d\xi\\
\leq&[c(\eta)N(t)]^{-4}E(u).
\end{align*}
Therefore, we obtain
\begin{equation*}
0\leq M(u)\lesssim_u
c(\eta)^{-4}N(t)^{-4}+\eta^\frac{\varepsilon}{2+\varepsilon}
\end{equation*}
for all $t\in\Bbb R$. As $u$ is a low-to-high cascade, there is a
sequence of times $t_n\rightarrow\infty$ so that
$N(t_n)\rightarrow\infty$. As $\eta>0$ is arbitrary, we may conclude
$M(u)=0$ and $u\equiv0$. This concludes the fact that
$S_I(u)=\infty$, thus settling Theorem \ref{lth}.
\end{proof}
\section{Soliton-like solutions}
In this section, we preclude the soliton-like solutions, namely, we
prove
\begin{theorem}[Absence of solitons]\label{soliton} Let $d\geq9$. There are no
global solutions to (\ref{fnls}) that are solitons in the sense of
Theorem \ref{enemies}.
\end{theorem}
First we prove that the potential cannot be very small.
\begin{proposition}[Potential energy bounded from below]\label{pe} Let $u$ be
the soliton-like solutions in the sense of Theorem \ref{enemies},
then we have
\begin{equation*}
\inf_{t\in\Bbb R}\big\|u(t,x)\big\|_{L_x^\frac{2d}{d-4}}>0.
\end{equation*}
\end{proposition}
\begin{proof} Suppose for a contradiction that $\inf_{t\in\Bbb
R}\big\|u(t,x)\big\|_{L_x^\frac{2d}{d-4}}=0$. Then there exists a
sequence $\{t_n\}$ such that
\begin{equation}\label{equ61}
\lim_{n\rightarrow\infty}\big\|u(t_n,x)\big\|_{L_x^\frac{2d}{d-4}}=0,
\end{equation}
where $t_n\rightarrow0$ (up to time translation) or
$t_n\rightarrow\pm\infty$.

Since $u\in C_t^0(\Bbb R, \dot{H}^2_x(\Bbb R^d))$, for any
$\varepsilon>0$, there exists an interval $\tilde{I}$
($\tilde{I}=(a, +\infty)$ if $t_n\rightarrow+\infty$;
$\tilde{I}=(-\infty, b)$ if $t_n\rightarrow-\infty$), such that
\begin{equation}
\big\|u(t,x)\big\|_{L_t^\infty
L_x^\frac{2d}{d-4}(\tilde{I}\times\Bbb R^d)}<\varepsilon.
\end{equation}
Using Lemma \ref{duh} and Strichartz estimates, we have
\begin{align*}
\big\|u\big\|_{L_t^2L_x^\frac{2d}{d-4}(\tilde{I}\times\Bbb
R^d)}\lesssim&\
\big\|i\int_t^{+\infty}e^{i(t-t')\Delta^2}(|u|^\frac{8}{d-4}u)(t')dt'\big\|_{L_t^2L_x^\frac{2d}{d-4}(\tilde{I}\times\Bbb
R^d)}\\
\lesssim&\
\big\||u|^\frac{8}{d-4}u\big\|_{L_t^2L_x^\frac{2d}{d+4}(\tilde{I}\times\Bbb
R^d)}\\
\lesssim&\
\big\|u\big\|_{L_t^2L_x^\frac{2d}{d-4}(\tilde{I}\times\Bbb
R^d)}\big\|u\big\|_{L_t^\infty
L_x^\frac{2d}{d-4}(\tilde{I}\times\Bbb R^d)}^\frac{8}{d-4}\\
\lesssim&\
\varepsilon^\frac{8}{d-4}\big\|u\big\|_{L_t^2L_x^\frac{2d}{d-4}(\tilde{I}\times\Bbb
R^d)}.
\end{align*}
Thus we get that $u\equiv0$ on $\tilde{I}$. By energy conservation,
$u\equiv0$ on $\Bbb R$, which contradicts $S_{\Bbb R}(u)=\infty$.
\end{proof}

\begin{proposition}[Concentration of $L^p$ norm] \label{con}
Let $u$ be the soliton-like solution in the sense of Theorem
\ref{enemies}, then for every $1<p<+\infty$, we have
\begin{equation*}
\inf_{t\in\Bbb R}\|u(t,x)\|_{L_x^p}>0.
\end{equation*}
\end{proposition}
\begin{proof} By Theorem \ref{neg}, $u\in L_x^2(\Bbb R^d)$. If
$p>2^{\#}$, then interpolation
\begin{equation*}
\|u(t)\|_{L_x^{2^{\#}}}\lesssim
\|u(t)\|_{L_x^2}^{1-\frac{4p}{(p-2)d}}\|u(t)\|_{L_x^p}^\frac{4p}{(p-2)d},
\end{equation*}
combined with Proposition \ref{pe}, yields that
\begin{equation*}
\inf_{t\in\Bbb R}\|u(t)\|_{L_x^p}>0.
\end{equation*}
If $1\leq p<2$, by interpolation
\begin{equation*}
\|u(t)\|_{L_x^2}\lesssim
\|u(t)\|_{L_x^p}^\frac{4p}{2d-(d-4)p}\|u(t)\|_{L^{2^{\#}}}^\frac{(2-p)d}{2d-(d-4)p}
\end{equation*}
and mass conservation, we have
\begin{equation*}
\inf_{t\in\Bbb R}\|u(t)\|_{L_x^p}>0.
\end{equation*}
Finally we consider the case of $2<p<2^{\#}$. If $\inf_{t\in \Bbb
R}\|u(t,x)\|_{L_x^p}=0$, then there exists $\{t_n\}$ such that
$\lim_{n\rightarrow\infty}\|u(t_n,x)\|_{L_x^p}=0$. On the other
hand, by (\ref{equ52}), (\ref{cc}) and Proposition \ref{pe}, we have
\begin{equation}\label{equ62}
\lim_{n\rightarrow\infty}\big\|P_{c(\eta)<\cdot<C(\eta)}u(t_n)\big\|_{L_x^{2^{\#}}}\gtrsim
1,
\end{equation}
as long as $\eta$ is chosen sufficiently small. By Sobolev
embedding,
\begin{align*}
\big\|P_{c(\eta)<\cdot<C(\eta)}u(t_n)\big\|_{L_x^{2^{\#}}}\lesssim&\big\||\nabla|
^{d(\frac{1}{p}-\frac{d-4}{2d})}P_{c(\eta)<\cdot<C(\eta)}u(t_n)\big\|_{L_x^p}\\
\lesssim& C(\eta)\|u(t_n)\|_{L_x^p}\rightarrow0, \ \ \text{as}\ \
n\rightarrow\infty.
\end{align*}
This contradicts (\ref{equ62}), hence completes the proof.
\end{proof}

To kill the soliton, we need the interaction Morawetz estimate. The
Interaction Morawetz estimates was obtained in \cite{Pa2} in
dimension $d\geq7$ and then was extended to dimensions $d\geq5$ in
\cite{mwz}.
\begin{proposition}[Interaction Morawetz estimates, \cite{mwz}]\label{ime} Let $u\in C_t^0(I, H^2_x(\Bbb R^d))$ be the
solution to
\begin{equation*}\label{fnlsp} \left\{ \aligned
    iu_t +  \Delta^2 u  & = \lambda|u|^{p-1}u, \quad  \text{in}\  \mathbb{R} \times \mathbb{R}^d,\\
     u(0)&=u_0(x), \quad\quad \text{in} \ \mathbb{R}^d.
\endaligned
\right.
\end{equation*}
where $1< p\leq 2^{\#}-1$. Then if $d>5$, we have
\begin{equation}\label{equ63}
\int_{I}\iint_{\Bbb R^d\times\Bbb
R^d}\frac{|u(t,x)|^2|u(t,y)|^2}{|x-y|^5}dxdydt+\int_{I}\iint_{\Bbb
R^d\times\Bbb
R^d}\frac{|u(t,y)|^2|u(t,x)|^\frac{2d}{d-4}}{|x-y|}dxdydt\lesssim_u
1.
\end{equation}
If $d=5$, we have
\begin{equation}\label{equ64}
\int_{I}\int_{\Bbb R^5}|u(t,x)|^4dxdt+\int_{I}\iint_{\Bbb
R^5\times\Bbb
R^5}\frac{|u(t,y)|^2|u(t,x)|^{10}}{|x-y|}dxdydt\lesssim_u 1.
\end{equation}
\end{proposition}
Proposition \ref{ime} and Theorem \ref{reg} yield
\begin{corollary}\label{co8}Fix $d\geq9$. Suppose $u$ is the
soliton-like solution to (\ref{fnls}) in the sense of Theorem
\ref{enemies}, then we have
\begin{equation}\label{equ633}
\int_{\Bbb R}\iint_{\Bbb R^d\times\Bbb
R^d}\frac{|u(t,x)|^2|u(t,y)|^2}{|x-y|^5}dxdydt+\int_{\Bbb
R}\iint_{\Bbb R^d\times\Bbb
R^d}\frac{|u(t,y)|^2|u(t,x)|^\frac{2d}{d-4}}{|x-y|}dxdydt\lesssim_u
1.
\end{equation}
\end{corollary}

Now we can kill the soliton thus complete the proof the Theorem 1.1.

{\it Proof of Theorem \ref{soliton}.} Fix $d\geq9$. Let $u$ be the
soliton-like solution as in Theorem \ref{enemies}. Then by Corollary
\ref{co8}, we have
\begin{equation*}
\big\||\nabla|^{-\frac{d-5}{2}}|u|^2\big\|_{L_{t,x}^2(\Bbb
R\times\Bbb R^d)}\lesssim_u 1.
\end{equation*}
On the other hand, by Sobolev embedding and energy conservation, we
have
\begin{equation*}
\big\|\nabla|u|^2\big\|_{L_t^\infty L_x^\frac{d}{d-3}(\Bbb
R\times\Bbb R^d)}\leq C\|\nabla u\|_{L_t^\infty
L_x^\frac{2d}{d-2}(\Bbb R\times\Bbb R^d)}\|u\|_{L_t^\infty
L_x^\frac{2d}{d-4}(\Bbb R\times\Bbb R^d)}\lesssim_u 1.
\end{equation*}
Therefore, by interpolation, we have
\begin{equation*}
\||u|^2\|_{L_t^{d-3}L_x^\frac{d(d-3)}{d^2-7d+15}(\Bbb R\times\Bbb
R^d)}\leq
C\big\||\nabla|^{-\frac{d-5}{2}}|u|^2\big\|_{L_{t,x}^2(\Bbb
R\times\Bbb R^d)}^\frac{2}{d-3}\big\|\nabla|u|^2\big\|_{L_t^\infty
L_x^\frac{d}{d-3}(\Bbb R\times\Bbb R^d)}^\frac{d-5}{d-3},
\end{equation*}
hence
\begin{equation}\label{equ65}
\big\|u\big\|_{L_t^{2(d-3)}L_x^\frac{2d(d-3)}{d^2-7d+15}(\Bbb
R\times\Bbb R^d)}\lesssim_u 1.
\end{equation}
However, by Proposition \ref{con},
\begin{equation*}
\big\|u\big\|_{L_x^\frac{2d(d-3)}{d^2-7d+15}(\Bbb R^d)}\gtrsim_u 1.
\end{equation*}
So
\begin{equation*}
\big\|u\big\|_{L_t^{2(d-3)}L_x^\frac{2d(d-3)}{d^2-7d+15}(\Bbb
R\times\Bbb R^d)}=+\infty,
\end{equation*}
which contradicts (\ref{equ65}). This completes the proof of Theorem
\ref{soliton}.

\vskip0.5cm
 \textbf{Acknowledgements:} C. Miao, G.Xu and L. Zhao
were partly supported by the NSF of China (No.10725102, No.10801015
and No. 10901148). The authors would like to thank B. Pausader for
providing us with his thesis.

\end{document}